\documentclass[11pt,draft,reqno]{amsart}

\usepackage{amssymb} 
\usepackage{dsfont} 
\usepackage{mathrsfs} 
\usepackage{stmaryrd} 
\usepackage{color}

\DeclareMathAlphabet{\mathpzc}{OT1}{pzc}{m}{it} 

\newtheorem{Thm}{Theorem}[section]

\newtheorem{Lem}{Lemma}[section]
\newtheorem{Prop}{Proposition}[section]

\newtheorem{Def}{Definition}[section]

\theoremstyle{definition}
\newtheorem{Rem}{Remark}[section]

\theoremstyle{definition}

\makeatletter
\@addtoreset{equation}{section}
\makeatother

\newcommand\setmeno{\!\smallsetminus\!} 
\newcommand\function{\longrightarrow} 
\newcommand\indicator{\mathds{1}} 
 \newcommand\en{\mathbb{N}} 
\newcommand\ar{\mathbb{R}} 
\newcommand{\eps}{\varepsilon} 
\providecommand{\clint}[1]{\hspace{0.045ex}\left[#1\right]} 
\providecommand{\clsxint}[1]{\hspace{0.1ex}\left[#1\right[\hspace{0.15ex}} 
\providecommand{\opint}[1]{\hspace{0.15ex}\left]#1\right[\hspace{0.15ex}} 

\DeclareMathOperator{\Id}{{\textsl{Id}}} 

\newcommand{\borel}{\mathscr{B}} 
\renewcommand{\L}{{\textsl{L}\hspace{0.17ex}}} 
\newcommand\leb{\mathpzc{L}} 
\DeclareMathOperator{\de}{d \! \hspace{0.2ex}} 
\newcommand{\Step}{{\textsl{St}\hspace{0.17ex}}} 

\newcommand\X{\textsl{X}\hspace{0.21ex}} 
\renewcommand\d{\textsl{d}} 
\newcommand\As{\textsl{A}\hspace{0.21ex}} 
\DeclareMathOperator{\Int}{int} 
\DeclareMathOperator{\Cl}{cl} 
\newcommand\Y{\textsl{Y}\hspace{0.21ex}} 
\renewcommand\S{\textsl{S}\hspace{0.21ex}} 

\DeclareMathOperator{\cont}{Cont}  
\DeclareMathOperator{\discont}{Discont}  
\newcommand{\Czero}{{\textsl{C}\hspace{0.18ex}}} 
\DeclareMathOperator{\Lipcost}{Lip} 
\newcommand{\Lip}{{\textsl{Lip}\hspace{0.15ex}}} 

\newcommand\E{\textsl{E}\hspace{0.21ex}} 
\newcommand\norm[2]{\Vert #1\Vert_{#2}} 
\newcommand{\convergedeb}{\rightharpoonup} 

\renewcommand\H{\mathcal{H}} 
\newcommand\duality[2]{\langle #1,#2 \rangle} 
\newcommand{\Conv}{\mathscr{C}} 
\newcommand\K{\mathcal{K}} 
\DeclareMathOperator{\Proj}{Proj} 
\DeclareMathOperator{\seg}{seg}

\newcommand\hausd{\mathpzc{H}} 
\newcommand\A{\mathcal{A}} 
\newcommand\B{\mathcal{B}} 
\newcommand\C{\mathcal{C}} 

\newcommand\Z{\mathcal{Z}} 

\newcommand{\BV}{{\textsl{BV}\hspace{0.17ex}}} 
\DeclareMathOperator{\pV}{V} 
\renewcommand\r{\textsl{r}} 
\newcommand\vartot[1]{\!\left\bracevert\! #1 \!\right\bracevert\!} 
\DeclareMathOperator{\D}{D\!} 

\newcommand{\ftilde}{\widetilde{f}} 
\newcommand{\utilde}{\widetilde{u}} 

\renewcommand{\P}{{\mathsf{P}}} 
\newcommand{\Pbar}{\overline{\P}} 

\renewcommand\sp{\hspace{3.1ex}} 
\providecommand{\cldxint}[1]{\hspace{0.15ex}\left]#1\right]} 
\renewcommand\l{\textsl{l}} 
\newcommand\lduality[2]{\left\langle #1,#2 \right\rangle} 

\definecolor{blu}{rgb}{0.1,0.1,1}

\definecolor{green}{rgb}{0.0, 0.5, 0.0}

\definecolor{marr}{rgb}{0.63, 0.47, 0.35}

\begin{document}


\title[Sweeping processes]{Sweeping processes with prescribed behaviour on jumps}

\author{Vincenzo Recupero and Filippo Santambrogio}
\thanks{Supported by the INdAM - GNAMPA Project 2016 
``Sweeping processes: teoria, controllo e applicazioni a problemi rate independent''.
}

\address{\textbf{Vincenzo Recupero}\\
        Dipartimento di Scienze Matematiche\\ 
        Politecnico di Torino\\
        Corso Duca degli Abruzzi 24\\ 
        I-10129 Torino\\ 
        Italy. \newline
        {\rm E-mail address:}
        {\tt vincenzo.recupero@polito.it}}

\address{\textbf{Filippo Santambrogio}\\
        Laboratoire de Math\'ematiques d'Orsay\\ 
        Univ. Paris-Sud, CNRS, Universit\'e Paris-Saclay,
        F-91405 Orsay cedex,\\ 
        France. \newline
        {\rm E-mail address:}
        {\tt filippo.santambrogio@math.u-psud.fr}}




\begin{abstract}
We present a generalized formulation of sweeping process where the behaviour of the solution is prescribed at the jump 
points of the driving moving set. An existence and uniqueness theorem for such formulation is proved. As a consequence we derive a formulation and an existence/uniqueness theorem for sweeping processes driven by an arbitrary $\BV$ moving set, whose evolution is not necessarily right continuous. Applications to the play operator of elastoplasticity are also shown.
\end{abstract}


\maketitle


\thispagestyle{empty}


\section{Introduction}

Sweeping processes are a class of evolution differential inclusions introduced by J.J. Moreau in a series of articles 
\cite{Mor71, Mor72, Mor73, Mor74} which culminated in the celebrated paper \cite{Mor77}, originating a research that is still active. The original formulation introduced in \cite{Mor71, Mor74} reads as follows. Let $\H$ be a real Hilbert space, $a, b \in \ar$, $a < b$, and, for every time $t \in \clint{a,b}$, let $\C(t)$ be a given nonempty, closed and convex subset of $\H$ such that the mapping $t \longmapsto \C(t)$ is Lipschitz continuous when the family of closed subsets of $\H$ is endowed with the Hausdorff metric. One has to find a  Lipschitz continuous function $y : \clint{a,b} \function \H$ such that
\begin{alignat}{3}
  & y(t) \in \C(t) & \quad & \forall t \in \clint{a,b}, \label{y in C - Lip - intro} \\
  & -y'(t) \in N_{\C(t)}(y(t)) & \quad & \text{for $\leb^{1}$-a.e. $t \in \clint{a,b}$}, \label{diff. incl. - Lip - intro} \\
  & y(0) = y_{0}, &  \label{in. cond. - Lip - intro}
\end{alignat}
$y_0$ being a prescribed point in $\C(a)$. Here $\leb^1$ is the Lebesgue measure, and $N_{\C(t)}$ is the exterior normal cone to $\C(t)$ at $y(t)$ (all the precise definitions will be given in Section \ref{S:Preliminaries}). When the interior of $\C(t)$ is nonempty, the process defined by \eqref{y in C - Lip - intro}-\eqref{in. cond. - Lip - intro} has a nice and useful geometrical-mechanical interpretation which we recall from \cite{Mor77}: ``the moving point $u(t)$ remains at rest as long as it happens to lie in the interior of $\C(t)$; when caught up with the boundary of the moving set, it can only proceed in an inward normal direction, as if pushed by this boundary, so as to go on belonging to $\C(t)$''. Moreau was originally motivated by plasticity and friction dynamics (cf. \cite{Mor74, Mor76b, Mor02}), but now sweeping processes have found applications to nonsmooth mechanics (see, e.g., \cite{Mon93, KunMon00, Pao16}), to economics (cf., e.g., \cite{Hen73, Cor83, FlaHirJou09}), to electrical circuits (see, e.g., 
\cite{AddAdlBroGoe07, BroThi10, AcaBonBro11, AdlHadThi14}), to crowd motion modeling (cf., e.g., 
\cite{MauVen08, MauRouSan10, MauVen11, MauRouSanVen11, MauRouSan14, DimMauSan16}), and to other fields (see, e.g., the references in the recent paper \cite{Thi16}).

In \cite{Mor77} the formulation \eqref{y in C - Lip - intro}-\eqref{in. cond. - Lip - intro} is extended to the case when the mapping $t \longmapsto \C(t)$ is of bounded variation and right continuous in the following natural way: it is proved that there is a unique $y \in \BV^\r(\clint{a,b};\H)$, the space of right continuous $\H$-valued functions of bounded variation, such that there exist a positive measure $\mu$ and a $\mu$-integrable function $v : \clint{a,b} \function \H$ satisfying
\begin{alignat}{3}
  & y(t) \in \C(t) & \quad &  \forall t \in \clint{a,b}, \label{y in C - BV - intro} \\
  & \! \D y = v \mu, & \label{Dy = w mu - intro} \\
  & -v(t) \in N_{\C(t)}(y(t)) & \quad & \text{for $\mu$-a.e. $t \in \clint{a,b}$}, \label{diff. incl. - BV - intro} \\
  & y(a) = y_{0}, &  \label{in. cond. - BV - intro}
\end{alignat}
where $\D y$ denotes distributional derivative of $y$, which is a measure since $y\in BV$. A relevant particular case is provided by the case when $\C(t) = u(t) - \Z$, with $u \in \BV^\r(\clint{0,T};\H)$, and $\Z \subseteq \H$ closed, convex and nonempty, namely the sweeping process driven by a set with constant shape. The resulting solution operator 
$\P : \BV^\r(\clint{a,b};\H) \function \BV^\r(\clint{a,b};\H)$ associating with $u$ the unique function $y$ satisfying 
\eqref{y in C - BV - intro}-\eqref{in. cond. - BV - intro} with $\C(t) = u(t) - \Z$ is called \emph{vector play operator} 
(see, e.g., \cite{Kre91, Kre97, KreLau02, Rec08, Rec11, Rec15a}) and has an important role in elasto-plasticity and hysteresis (cf., e.g., \cite{KraPok89, Vis94, BroSpr96, Kre97, Mie05, MieRou15}). 

The theoretical analysis of problem \eqref{y in C - BV - intro}-\eqref{in. cond. - BV - intro} has been expanded in various directions: the case of $\C$ continuous was first dealt in \cite{Mon93}, the nonconvex case has been studied in several papers, e.g. 
\cite{Val88, Val90, Val92, CasDucVal93, CasMon96, ColGon99, Ben00, ColMon03, Thi03, BouThi05, EdmThi06, Thi08, HasJouThi08, BerVen10, SenThi14};  
for stochastic versions see, e.g., \cite{Cas73, Cas76, BerVen11, CasMonRay16}, while periodic solutions can be found in \cite{CasMon95}. The continuous dependence properties of various sweeping problems are investigated, e.g., in \cite{Mor77, BroKreSch04, Rec09, KreRoc11, Rec11, Rec11b, KleRec16, KopRec16}, and the control problems are
studied, e.g., in \cite{ColHenHoaMor15, ColHenNguMor16, ColPal16}.

When the moving set $\C$ jumps, the geometrical interpretation of the sweeping process 
\eqref{y in C - BV - intro}-\eqref{in. cond. - BV - intro} has to be revisited by analyzing the behaviour of the solution $y$ at 
jump points $t$ of $\C$: at such points it can be showed (cf. \cite{Mor77}) that $y(t) = y(t+) = \Proj_{\C(t)}(u(t-))$, where 
$\Proj$ is the classical projection operator, thus $y$ instantaneously moves from $\C(t-)$ to $\C(t) = \C(t+)$ along the shortest path which allows to satisfy the constraint \eqref{y in C - BV - intro}. Although this is a very natural requirement of formulation \eqref{y in C - BV - intro}-\eqref{in. cond. - BV - intro}, this is not the only one: let us see, for instance, two cases where a different behaviour can be prescribed at the jump points of $\C$.

First let us consider the problem of extending the formulation \eqref{y in C - BV - intro}-\eqref{in. cond. - BV - intro} to the case of arbitrary moving sets $\C$ of bounded variation, not necessarily right continuous. Of course one could replace 
$\C(t)$ by $\C(t+)$ in the differential inclusions \eqref{diff. incl. - BV - intro}, but this would not take into account of the position of $\C(t)$ at a jump point $t$: it would be very natural, instead, to expect that for such $t$ one should find 
$y(t+) = \Proj_{\C(t+)}(\Proj_{\C(t)}(y(t-)))$. 

Another situation is given by the play operator $\P$: when a jump point $t$ of the ``input function'' 
$u \in \BV^\r(\clint{a,b};\H)$ occurs, it would be very natural to expect that $y(t+) $ would be the final point of the play operator driven by the segment $(1-\sigma)u(t-) + \sigma u(+)$, $\sigma \in \clint{0,1}$, in other words we would expect that $y$ behaves as if $u$ traversed the segment joining $u(t-)$ and $u(t+)$ with ``infinite velocity'' at every jump points $t$. Nevertheless it can be proved (cf. \cite[Section 5.3]{Rec11} and \cite{KreRec14}) that in general the solution 
$y = \P(u)$ provided by \eqref{y in C - BV - intro}-\eqref{in. cond. - BV - intro} does not satisfy this property, which would be also natural if we recall that the play operator is \emph{rate independent}, i.e. $\P(u \circ \psi) = \P(u) \circ \psi$ for every increasing surjective reparametrization of time $\psi : \clint{a,b} \function \clint{a,b}$.

The aim of this paper is to provide a general formulation which take into account of all these possible different behaviors at jumps point of the driving moving set $\C$. To be more precise if $t \longmapsto \C(t)$ is right continuous of bounded variation, we prove that there exists a unique function $y \in \BV^\r(\clint{a,b};\H)$ such that $y(t) \in \C(t) = \C(t+)$ for every $t \in \clint{a,b}$ and 
\begin{alignat}{3}
  & \! \D y = v \mu, & \label{Dy = w mu - intro - gen} \\
  & v(t) + N_{\C(t)}(y(t)) \ni 0 & \quad & \text{for $\mu$-a.e. $t \in \clint{a,b}$}, \label{diff. incl. - BV - intro - gen} \\
  & y(t+) = g_t(y(t-)) & \quad & \text{for every discontinuity point $t$ of $\C$}, \label{presc beh - intro} \\
  & y(0) = y_{0}, &  \label{in. cond. - BV - intro - gen}
\end{alignat}
where $g_t : \C(t-) \function \C(t)$ is a family of functions prescribing the behaviour of $y$ at every jump point $t$ of 
$\C$ (a sort of family of ``initial conditions at the jump points of the datum'', but we actually consider a more general situation). The case of the arbitrary moving set $\C$, not necessarily continuous, can be immediately deduced by taking $g_t = \Proj_{\C(t+)} \circ \Proj_{\C(t)}$, the ``double projection''. 

The paper is organized as follows. In the next section we present some technical preliminaries and in Section \ref{S:main results} we state our main result. This result will be proved in Section \ref{S:proofs} and finally in the last Section \ref{S:applications} we present some applications and consequences of our main results.


\section{Preliminaries}\label{S:Preliminaries}

In this section we recall the main definitions and tools needed in the paper. The set of integers greater than or equal to 
$1$ will be denoted by $\en$. Given an interval $I$ of the real line $\ar$, if $\borel(I)$ indicates the family of Borel sets in 
$I$, $\mu : \borel(I) \function \clint{0,\infty}$ is a measure, $p \in \clint{1,\infty}$, and $\textsl{E}$ is a Banach space, then 
the space of $\E$-valued functions which are $p$-integrable with respect to $\mu$ will be denoted by $\L^p(I, \mu; \E)$ or simply by $\L^p(\mu; \E)$. We do not identify two functions which are equal $\mu$-almost everywhere. The one 
dimensional Lebesgue measure is denoted by $\leb^1$ and the Dirac delta at a point $t \in \ar$ is denoted by 
$\delta_t$. For the theory of integration of vector valued functions we refer, e.g., to \cite[Chapter VI]{Lan93}.

\subsection{Functions with values in a metric space}

In this subsection we assume that 
\begin{equation}\label{X complete metric space}
  \text{$(\X,\d)$ is a complete metric space},
\end{equation}
where we admit that $\d$ is an extended metric, i.e. $\X$ is a set and $\d : \X \times \X \function \clint{0,\infty}$ satisfies 
the usual axioms of a distance, but may take on the value $\infty$. The notion of completeness remains unchanged. The general topological notions of interior, closure and boundary of a subset $\As \subseteq \X$ will be respectively denoted by $\Int(\As)$, $\Cl(\As)$ and $\partial \As$. We also set $\d(x,\As) := \inf_{a \in \As} \d(x,a)$. If $(\Y,\d_\Y)$ is a metric space then the continuity set of a function $f : \Y \function \X$ is denoted by $\cont(f)$, while
$\discont(f) := T \setmeno \cont(f)$. The set of continuous $\X$-valued functions defined on $\Y$ is denoted by 
$\Czero(\Y;\X)$. For $\S \subseteq \Y$ we write $\Lipcost(f,\S) := \sup\{\d(f(s),f(t))/\d_\Y(t,s)\ :\ s, t \in \S,\ s \neq t\}$, 
$\Lipcost(f) := \Lipcost(f,\Y)$, the Lipschitz constant of $f$, and 
$\Lip(\Y;\X) := \{f : \Y \function \X\ :\ \Lipcost(f) < \infty\}$, the set of $\X$-valued Lipschitz continuous functions on 
$\Y$. If $\E$ is a Banach space with norm $\norm{\cdot}{\E}$ and $\S \subseteq \Y$, we set 
$\norm{f}{\infty,\S} := \sup_{t \in \S}\norm{f(t)}{\E}$ for every function $f : \Y \function \E$ We recall now the notion of 
$\BV$ function with values in a metric space (see, e.g., \cite{Amb90, Zie94}).

\begin{Def}
Given an interval $I \subseteq \ar$, a function $f : I \function \X$, and a subinterval $J \subseteq I$, the 
\emph{(pointwise) variation of $f$ on $J$} is defined by
\begin{equation}\notag
  \pV(f,J) := 
  \sup\left\{
           \sum_{j=1}^{m} \d(f(t_{j-1}),f(t_{j}))\ :\ m \in \en,\ t_{j} \in J\ \forall j,\ t_{0} < \cdots < t_{m} 
         \right\}.
\end{equation}
If $\pV(f,I) < \infty$ we say that \emph{$f$ is of bounded variation on $I$} and we set 
$\BV(I;\X) := \{f : I \function \X\ :\ \pV(f,I) < \infty\}$.
\end{Def}

It is well known that the completeness of $\X$ implies that every $f \in \BV(I;\X)$ admits one-sided limits $f(t-), f(t+)$ at every point $t \in I$, with the convention that $f(\inf I-) := f(\inf I)$ if $\inf I \in I$, and $f(\sup I+) := f(\sup I)$ if 
$\sup I \in I$, and that $\discont(f)$ is at most countable. We set 
$\BV^\r(I;\X) := \{f \in \BV(I;\X)\ :\ f(t) = f(t+) \quad \forall t \in I\}$ and if $I$ is bounded we have 
$\Lip(I;\X) \subseteq \BV(I;\X)$.

\subsection{Convex sets in Hilbert spaces}

Throughout the remainder of the paper we assume that
\begin{equation}\label{H-prel}
\begin{cases}
  \text{$\H$ is a real Hilbert space with inner product $(x,y) \longmapsto \duality{x}{y}$} \\
  \norm{x}{} := \duality{x}{x}^{1/2}
\end{cases},
\end{equation}
and we endow $\H$ with the natural metric defined by $\d(x,y) := \norm{x-y}{}$, $x, y \in \H$. We set
\begin{equation}\notag
  \Conv_\H := \{\K \subseteq \H\ :\ \K \ \text{nonempty, closed and convex} \}.
\end{equation}
If  $\K \in \Conv_\H$ and $x \in \H$, then $\Proj_{\K}(x)$ is the projection on $\K$, i.e. $y = \Proj_\K(x)$ is the unique 
point such that $d(x,\K) = \norm{x-y}{}$, and it is also characterized by the two conditions
\begin{equation}\notag
  y \in \K, \quad \duality{x - y}{v - y} \le 0 \qquad \forall v \in \K.
\end{equation}
If $\K \in \Conv_\H$ and $x \in \K$, then $N_\K(x)$ denotes the \emph{(exterior) normal cone of $\K$ at $x$}:
\begin{equation}\label{normal cone}
  N_\K(x) := \{u \in \H\ :\ \duality{u}{v - x} \le 0\ \forall v \in \K\} = \Proj_\K^{-1}(x) - x.
\end{equation}
It is well known that the multivalued mapping $x \longmapsto N_\K(x)$ is \emph{monotone}, i.e. 
$\duality{u_1 - u_2}{x_1 - x_2} \ge 0$ whenever $x_j \in \K$, $u_j \in N_\K(x_j)$, $j = 1, 2$ (see, e.g., 
\cite[Exemple 2.8.2, p.46]{Bre73}). We endow the set $\Conv_{\H}$ with the Hausdorff distance. Here we recall the definition.

\begin{Def}
The \emph{Hausdorff distance} $\d_\hausd : \Conv_{\H} \times \Conv_{\H} \function \clint{0,\infty}$ is defined by
\begin{equation}\notag
  \d_{\hausd}(\A,\B) := 
  \max\left\{\sup_{a \in \A} \d(a,\B)\,,\, \sup_{b \in \B} \d(b,\A)\right\}, \qquad
  \A, \B \in \Conv_{\H}.
\end{equation}
\end{Def}
The metric space $(\Conv_\H, \d_\hausd)$ is complete (cf. \cite[Theorem II-14, Section II.3.14, p. 47]{CasVal77}).


\subsection{Differential measures}

We recall that a \emph{$\H$-valued measure on $I$} is a map $\mu : \borel(I) \function \H$ such that 
$\mu(\bigcup_{n=1}^{\infty} B_{n})$ $=$ $\sum_{n = 1}^{\infty} \mu(B_{n})$ whenever $(B_{n})$ is a sequence of mutually disjoint sets in $\borel(I)$. The \emph{total variation of $\mu$} is the positive measure 
$\vartot{\mu} : \borel(I) \function \clint{0,\infty}$ defined by
\begin{align}\label{tot var measure}
  \vartot{\mu}(B)
  := \sup\left\{\sum_{n = 1}^{\infty} \norm{\mu(B_{n})}{}\ :\ 
                 B = \bigcup_{n=1}^{\infty} B_{n},\ B_{n} \in \borel(I),\ 
                 B_{h} \cap B_{k} = \varnothing \text{ if } h \neq k\right\}. \notag
\end{align}
The vector measure $\mu$ is said to be \emph{with bounded variation} if $\vartot{\mu}(I) < \infty$. In this case the equality $\norm{\mu}{} := \vartot{\mu}(I)$ defines a norm on the space of measures with bounded variation (see, e.g. 
\cite[Chapter I, Section  3]{Din67}). 

If $\nu : \borel(I) \function \clint{0,\infty}$ is a positive bounded Borel measure and if $g \in \L^1(I,\nu;\H)$, then $g\nu$ will denote the vector measure defined by $g\nu(B) := \int_B g\de \nu$ for every $B \in \borel(I)$. In this case 
$\vartot{g\nu}(B) = \int_B \norm{g(t)}{}\de \nu$ for every $B \in \borel(I)$ (see \cite[Proposition 10, p. 174]{Din67}). Moreover a vector measure $\mu$ is called \emph{$\nu$-absolutely continuous} if $\mu(B) = 0$ whenever 
$B \in \borel(I)$ and $\nu(B) = 0$. 

Assume that $\mu : \borel(I) \function \H$ is a vector measure with bounded variation and let $f : I \function \H$ and 
$\phi : I \function \ar$ be two \emph{step maps with respect to $\mu$}, i.e. there exist $f_{1}, \ldots, f_{m} \in \H$, 
$\phi_{1}, \ldots, \phi_{m} \in \H$ and $A_{1}, \ldots, A_{m} \in \borel(I)$ mutually disjoint such that 
$\vartot{\mu}(A_{j}) < \infty$ for every $j$ and $f = \sum_{j=1}^{m} \indicator_{A_{j}} f_{j},$, 
$\phi = \sum_{j=1}^{m} \indicator_{A_{j}} \phi_{j},$ where $\indicator_{S} $ is the characteristic function of a set $S$, i.e. 
$\indicator_{S}(x) := 1$ if $x \in S$ and $\indicator_{S}(x) := 0$ if $x \not\in S$. For such step functions we define 
$\int_{I} \duality{f}{\mu} := \sum_{j=1}^{m} \duality{f_{j}}{\mu(A_{j})} \in \ar$ and
$\int_{I} \phi \de \mu := \sum_{j=1}^{m} \phi_{j} \mu(A_{j}) \in \H$. If $\Step(\vartot{\mu};\H)$ (resp. $\Step(\vartot{\mu})$) is the set of $\H$-valued (resp. real valued) step maps with respect to $\mu$, then the maps
$\Step(\vartot{\mu};\H)$ $\function$ $\H : f \longmapsto \int_{I} \duality{f}{\mu}$ and
$\Step(\vartot{\mu})$ $\function$ $\H : \phi \longmapsto \int_{I} \phi \de \mu$ 
are linear and continuous when $\Step(\vartot{\mu};\H)$ and $\Step(\vartot{\mu})$ are endowed with the 
$\L^{1}$-seminorms $\norm{f}{\L^{1}(\vartot{\mu};\H)} := \int_I \norm{f}{} \de \vartot{\mu}$ and
$\norm{\phi}{\L^{1}(\vartot{\mu})} := \int_I |\phi| \de \vartot{\mu}$. Therefore they admit unique continuous extensions 
$\mathsf{I}_{\mu} : \L^{1}(\vartot{\mu};\H) \function \ar$ and 
$\mathsf{J}_{\mu} : \L^{1}(\vartot{\mu}) \function \H$,
and we set 
\[
  \int_{I} \duality{f}{\de \mu} := \mathsf{I}_{\mu}(f), \quad
  \int_{I} \phi \mu := \mathsf{J}_{\mu}(\phi),
  \qquad f \in \L^{1}(\vartot{\mu};\H),\quad \phi \in \L^{1}(\vartot{\mu}).
\]

If $\nu$ is bounded positive measure and $g \in \L^{1}(\nu;\H)$, arguing first on step functions, and then taking limits, it is easy to check that $\int_I\duality{f}{\de(g\nu)} = \int_I \duality{f}{g}\de \nu$ for every $f \in \L^{\infty}(\mu;\H)$. The following results (cf., e.g., \cite[Section III.17.2-3, pp. 358-362]{Din67}) provide a connection between functions with bounded variation and vector measures which will be implicitly used in the paper.

\begin{Thm}\label{existence of Stietjes measure}
For every $f \in \BV(I;\H)$ there exists a unique vector measure of bounded variation $\mu_{f} : \borel(I) \function \H$ such that 
\begin{align}
  \mu_{f}(\opint{c,d}) = f(d-) - f(c+), \qquad \mu_{f}(\clint{c,d}) = f(d+) - f(c-), \notag \\ 
  \mu_{f}(\clsxint{c,d}) = f(d-) - f(c-), \qquad \mu_{f}(\cldxint{c,d}) = f(d+) - f(c+). \notag 
\end{align}
whenever $c < d$ and the left hand side of each equality makes sense. Conversely, if $\mu : \borel(I) \function \H$ is a vector measure with bounded variation, and if $f_{\mu} : I \function \H$ is defined by 
$f_{\mu}(t) := \mu(\clsxint{\inf I,t} \cap I)$, then $f_{\mu} \in \BV(I;\H)$ and $\mu_{f_{\mu}} = \mu$.
\end{Thm}

\begin{Prop}
Let $f  \in \BV(I;\H)$, let $g : I \function \H$ be defined by $g(t) := f(t-)$, for $t \in \Int(I)$, and by $g(t) := f(t)$, if 
$t \in \partial I$, and let $V_{g} : I \function \ar$ be defined by $V_{g}(t) := \pV(g, \clint{\inf I,t} \cap I)$. Then  
$\mu_{g} = \mu_{f}$ and $\vartot{\mu_{f}} = \mu_{V_{g}} = \pV(g,I)$.
\end{Prop}

The measure $\mu_{f}$ is called \emph{Lebesgue-Stieltjes measure} or \emph{differential measure} of $f$. Let us see the connection with the distributional derivative. If $f \in \BV(I;\H)$ and if $\overline{f}: \ar \function \H$ is defined by
\begin{equation}\label{extension to R}
  \overline{f}(t) :=
  \begin{cases}
    f(t) 	& \text{if $t \in I$} \\
    f(\inf I)	& \text{if $\inf I \in \ar$, $t \not\in I$, $t \le \inf I$} \\
    f(\sup I)	& \text{if $\sup I \in \ar$, $t \not\in I$, $t \ge \sup I$}
  \end{cases},
\end{equation}
then, as in the scalar case, it turns out (cf. \cite[Section 2]{Rec11}) that $\mu_{f}(B) = \D \overline{f}(B)$ for every 
$B \in \borel(\ar)$, where $\D\overline{f}$ is the distributional derivative of $\overline{f}$, i.e.
\[
  - \int_\ar \varphi'(t) \overline{f}(t) \de t = \int_{\ar} \varphi \de \D \overline{f} 
  \qquad \forall \varphi \in \Czero_{c}^{1}(\ar;\ar),
\]
$\Czero_{c}^{1}(\ar;\ar)$ being the space of real continuously differentiable functions on $\ar$ with compact support.
Observe that $\D \overline{f}$ is concentrated on $I$: $\D \overline{f}(B) = \mu_f(B \cap I)$ for every $B \in \borel(I)$, hence in the remainder of the paper, if $f \in \BV(I,\H)$ then we will simply write
\begin{equation}
  \D f := \D\overline{f} = \mu_f, \qquad f \in \BV(I;\H),
\end{equation}
and from the previous discussion it follows that 
\begin{equation}
  \norm{\D f}{} = \vartot{\D f}(I) = \norm{\mu_f}{}  = \pV(f,I) \qquad \forall f \in \BV^\r(I;\H).
\end{equation}


\section{Main result}\label{S:main results}

We are now in position to state the main theorem of the paper.

\begin{Thm}\label{main thm}
Assume that $-\infty < a < b < \infty$, $\C \in \BV^\r(\clint{a,b};\Conv_\H)$, $y_0 \in \C(a)$, $S \subseteq \cldxint{a,b}$, and that for every $t \in S$ we are given a function $g_t : \C(t-) \function \C(t)$ such that $\Lipcost(g_t) \le 1$ and
\begin{equation}\label{sum gt-mr}
  \sum_{t \in S} \norm{g_t - \Id}{\infty,\C(t-)} < \infty.
\end{equation}
Then there exists a unique $y \in \BV^\r(\clint{a,b};\H)$ such that there is a measure 
$\mu : \borel(\clint{a,b}) \function \clsxint{0,\infty}$ and a function $v \in \L^1(\mu;\H)$ for which
\begin{alignat}{3}
  & y(t) \in \C(t) & \quad & \forall t \in \clint{a,b}, \label{constr. [0,T[-mr} \\
  & \! \D y = v \mu, & \label{Dy =vmu [0,T[-mr} \\
  & -v(t) \in N_{\C(t)}(y(t)) & \quad & \text{for $\mu$-a.e. $t \in \clint{a,b} \setmeno S$}, \label{diff. incl. [0,T[-mr} \\
  & y(t) = g_t(y(t-)) & \quad & \forall t \in S, \label{cond. on S, [0,T[-mr} \\
  & y(a)  = y_0. & \label{i.c. [0,T[-mr}
\end{alignat}
Moreover if $a \le s < t \le b$ we have
\begin{equation}\label{estimate for Var(y)_S gen}
  \pV(y,\clint{s,t})  \le \pV(\C,\clint{s,t}) + 
  \sum_{r \in S \cap \clint{s,t}} \left(\norm{g_r - \Id}{\infty,\C(r-)} - \d_\hausd(\C(r-),\C(r))\right).
\end{equation}
Finally if $y_{0,j} \in \C(a)$, $j = 1, 2$ and $y_j$ is the only function such that there is a measure 
$\mu_j : \borel(\clsxint{a,b}) \function \clsxint{0,\infty}$ and a function $v_j \in \L^1(\mu;\H)$ for which
\eqref{constr. [0,T[-mr}-\eqref{i.c. [0,T[-mr} hold with $y, v, \mu, y_0$ replaced respectively by 
$y_j, v_j, \mu_j, y_{0,j}$, then 
\begin{equation}\label{cont. dep.-mr}
  t \longmapsto \norm{y_1(t) - y_2(t)}{}^2 \text{ is nonincreasing}. \notag
\end{equation}
\end{Thm}

In Section \ref{S:applications} we will show a series of results that can be deduced from Theorem \ref{main thm}.


\section{Proofs}\label{S:proofs}

We start by recalling the existence and uniqueness result of classical sweeping processes due to J.J. Moreau (cf. \cite{Mor77}).

\begin{Thm}\label{T:sweeping process existence}
If $-\infty < a < b < \infty$, $\C \in \BV^\r(\clsxint{a,b};\Conv_\H)$, and $y_0 \in \C(a)$, then there exists a unique 
$y \in \BV^\r(\clsxint{a,b};\H)$ such that there is a measure $\mu : \borel(\clsxint{a,b}) \function \clsxint{0,\infty}$ and a function $v \in \L^1(\mu;\H)$ for which
\begin{alignat}{3}
 & y(t) \in \C(t) & \quad & \forall t \in \clsxint{a,b}, \label{constr-classical}\\
 & \D y = v \mu, & \label{Dy decomposition-classical} \\
 &  -v(t) \in N_{\C(t)}(y(t)) & \quad & \text{for $\mu$-a.e. $t \in \opint{a,b}$}, \label{diff incl-classical} \\
 &  y(a) = y_0. & \label{i.c.-classical}
\end{alignat}
Moreover 
\begin{equation}\label{estimate for var(y)-classical}
  \pV(y,\clint{s,t}) \le \pV(\C,\clint{s,t}) \qquad \forall s, t \in \clsxint{a,b}, \ s < t.
\end{equation}
Finally if $y_{0,j} \in \C(a)$, $j = 1, 2$ and $y_j$ is the only function such that there is a measure 
$\mu_j : \borel(\clint{a,b}) \function \clsxint{0,\infty}$ and a function $v_j \in \L^1(\mu;\H)$ for which
\eqref{constr-classical}-\eqref{i.c.-classical} hold with $y, v, \mu, y_0$ replaced respectively by 
$y_j, v_j, \mu_j, y_{0,j}$, then 
\begin{equation}\label{moreau cont. dep.}
  t \longmapsto \norm{y_1(t) - y_2(t)}{}^2 \text{ is nonincreasing}.
\end{equation}
\end{Thm}

Concerning Theorem \ref{T:sweeping process existence} let us observe that the existence and uniqueness of the solution $y$ is proved in \cite[Proposition 3b]{Mor77}, while formula \eqref{estimate for var(y)-classical} is proved in \cite[Proposition 2c]{Mor77}. Finally the last statement is proved in \cite[Proposition 2b]{Mor77}.
 
It is important to compare Theorem \ref{T:sweeping process existence} to our main result, Theorem \ref{main thm}: it is possible to see that Theorem \ref{main thm} includes the statement of Theorem \ref{T:sweeping process existence} when 
$g_t: \C(t-) \function \C(t)$ is given by the projection onto $\C(t)$, since in this case 
$\norm{g_t - \Id}{\infty,\C(t-)}\le  \d_{\hausd}(\C(t-),\C(t))$. The goal of our main theorem is to allow for different prescribed behaviors at points $t\in S$, and the proof will be based on a suitable combination of Theorem \ref{T:sweeping process existence} with some explicit applications of the maps $g_t$.
  
In the following Lemma we provide an integral formulation of the sweeping process.

\begin{Lem}\label{L:int. charact.}
Assume that $-\infty < a < b < \infty$, $\mu : \borel(\clsxint{a,b}) \function \clsxint{0,\infty}$ is a measure, 
$B \in \borel(\clsxint{a,b})$, and $\mu(B) \neq 0$. If $\C \in \BV^\r(\clsxint{a,b};\Conv_\H)$, $v \in \L^1(\mu;\H)$, 
$y \in \BV^\r(\clsxint{a,b};\H)$, and $y(t) \in \C(t)$ for every $t \in \clsxint{a,b}$, then the following two conditions are equivalent.
\begin{itemize}
\item[(i)]
  $-v(t) \in N_{\C(t)}(y(t))$ for $\mu$-a.e. $t \in B$.
\item[(ii)]
  $\displaystyle{\int_{B} \duality{y(t) - z(t)}{v(t)} \de \mu(t) \le 0}$ for every $\mu$-measurable  
  $z : \clint{0,T} \function \H$ such that $z(t) \in \C(t)$ for every $t \in B$.
\end{itemize}
\end{Lem}

\begin{proof}
Let us start by assuming that (i) holds and let $z : \clsxint{a,b} \function \H$ be a $\mu$-measurable function such that 
$z(t) \subseteq \C(t)$ for every $t \in B$. Then it follows that
\[
  \duality{y(t) - z(t)}{v(t)} \le 0 \qquad \text{$\forall t \in B$},
\]
and integrating over $B$ we infer condition (ii). Now assume that (ii) is satisfied and observe that
$\mathfrak{V} = \{(t,\clint{t-h,t+h} \cap B)\ :\ h > 0,\ t \in B\}$ is a $\mu$-Vitali relation covering $B$ according to the definition given in \cite[Section 2.8.16, p. 151]{Fed69}. Recall that if $f \in \L^1(\mu, B; \H)$ then there exists a 
$\mu$-zero measure set $Z$ such that $f(B \setmeno Z)$ is separable (see, e.g., \cite[Property M11, p. 124]{Lan93}), therefore from (the proof) of \cite[Corollary 2.9.9., p. 156]{Fed69} it follows that 
\begin{equation}\label{Dl-Lebesgue points}
  \lim_{h \searrow 0} \frac{1}{\mu(\clint{t-h,t+h} \cap B)} \int_{\clint{t-h,t+h} \cap B} \norm{f(\tau) - f(t)}{E} \de \mu(\tau) 
  = 0 \qquad \text{for $\mu$-a.e. $t \in B$}.
\end{equation}
In \cite{Fed69} the points $s$ satisfying \eqref{Dl-Lebesgue points} are called \emph{$\mu$-Lebesgue points of $f$ on 
$B$ with respect to the $\mu$-Vitali relation $\mathfrak{V}$} covering $B$. Let $L$ be the set of $\mu$-Lebesgue points for both $\tau \longmapsto v(\tau)$ and $\tau \longmapsto \duality{y(\tau)}{v(\tau)}$ on $B$ with respect to 
$\mathfrak{V}$, fix $t \in B$, and choose $\zeta_t \in \C(t)$ arbitrarily. Thanks to \cite[Theorem III.9, p. 67]{CasVal77} there exists a $\mu$-measurable function $\zeta : \clsxint{a,b} \function \H$ such that $\zeta(t) = \zeta_t$ and 
$\zeta(\tau) \in \C(\tau)$ for every $\tau \in \clsxint{a,b}$, therefore a straighforward computation shows that 
\[ 
  \lim_{h \searrow 0}  
  \frac{1}{\mu(\clint{t-h,t+h} \cap B)} \int_{\clint{t-h,t+h} \cap B} \duality{\zeta(\tau)}{v(\tau)} \de \mu(\tau) = 
  \duality{\zeta_t}{v(t)}.
\]
The function 
$z(\tau) := \indicator_{\clsxint{a,b} \cap \clint{t-h,t+h}}(\tau) \zeta(\tau) + \indicator_{\clsxint{a,b} \setmeno \clint{t-h,t+h}}(\tau) y(\tau)$, 
$\tau \in \clsxint{a,b}$, is well defined for every sufficiently small $h > 0$, therefore $z(\tau) \in \C(\tau)$ for every 
$\tau \in \clsxint{a,b}$, thus taking $z$ in condition (iii) we get
\[
  \int_{\clint{t-h,t+h} \cap B} \duality{y(\tau)}{v(\tau)}\de \mu(\tau) \le 
  \int_{\clint{t-h,t+h} \cap B} \duality{\zeta_t}{v(\tau)} \de \mu(\tau).
\] 
Dividing this inequality by $\mu(\clint{t-h,t+h} \cap B)$ and taking the limit as $h \searrow 0$ we get
$\duality{y(t) - \zeta_t}{v(t)} \le 0$. Therefore we have proved that
\begin{equation}
  \duality{y(t) - z}{v(t)} \le 0 \quad \forall z \in \C(t), \quad \text{for $\mu$-a.e. $t \in B$,} \notag
\end{equation}
i.e. condition (i) holds.
\end{proof}

Now it is convenient to recall the notion of normalized arc-length parametrization for a metric-space-valued curve, provided by the following proposition (cf., e.g., \cite{Fed69, Rec11})
\begin{Prop}\label{D:reparametrization}
Assume that \eqref{X complete metric space} is satisfied and that $f \in \BV(\clint{a,b};\X)$. We define 
$\ell_f : \clint{a,b} \function \clint{a,b}$ by
\begin{equation}\notag
  \ell_f(t) := 
  \begin{cases}
    a + \dfrac{b-a}{\pV(f,\clint{a,b})}\pV(f,\clint{a,t}) & \text{if $\pV(f,\clint{a,b}) \neq 0$} 	\\
    															\\
    a 								           & \text{if $\pV(f,\clint{a,b}) = 0$}
  \end{cases}
    \qquad t \in \clint{a,b}.
\end{equation}
If $\X = \H$, a Hilbert space, and $f \in \BV^\r(\clint{a,b};\H)$, then there is a unique $\ftilde \in \Lip(\clint{a,b};\H)$ such that $\Lipcost(\ftilde) \le \pV(f,\clint{a,b})/(b-a)$ and 
\begin{align}
  & f = \utilde \circ \ell_f, \label{reparametrization1} 
  \\ 
  & \utilde(\ell_f(t-)(1-\lambda) + \ell_f(t)\lambda) =  
    (1-\lambda)f(t-) + \lambda f(t) 
     \qquad \forall t \in \clint{a,b},\ \forall \lambda \in \clint{0,1}. \label{reparametrization2}
\end{align}
\end{Prop}
The normalization factor in Definition \ref{D:reparametrization} is not necessary in the proofs of our theorems but it is consistent and simplifies the statement of Theorem \ref{T:Pbar}.

In the following lemma we show that we can take $\mu = \D\ell_\C$, and in this case we have an explicit bound for the density of $\D y$ with respect to $\D\ell_\C$.

\begin{Lem}\label{L:mu=Dell}
Assume that $-\infty < a < b < \infty$, $\C \in \BV^\r(\clsxint{a,b};\Conv_\H)$, $y_0 \in \C(c)$, and $y$ is the only function such that there is a measure $\mu : \borel(\clsxint{a,b}) \function \clsxint{0,\infty}$ and a function $v \in \L^1(\mu;\H)$ for which \eqref{constr-classical}-\eqref{i.c.-classical} hold. Then there is a unique (up to $\D\ell_\C$-equivalence) 
$w \in \L^1(\D\ell_\C;\H)$ such that 
\begin{equation}
  \D y = w \D \ell_\C,
\end{equation}
and we have 
\begin{align}
  & -w(t) \in N_{\C(t)}(y(t)) \qquad \text{for $\D\ell_\C$-a.e. $t \in \opint{a,b}$}, \label{diff. incl. for Dell} \\
  & \norm{w(t)}{} \le \pV(\C, \clint{a,b})/(b-a) \qquad \text{for $\D\ell_\C$-a.e. $t \in \clsxint{a,b}$}. 
      \label{|v| < 1-classical}
\end{align}
\end{Lem}

\begin{proof}
The existence of a measure $\mu : \borel(\clsxint{a,b}) \function \clsxint{0,\infty}$ and a function $v \in \L^1(\mu;\H)$ such that \eqref{constr-classical}--\eqref{i.c.-classical} hold is guaranteed by Theorem 
\ref{T:sweeping process existence}, which, by estimate \eqref{estimate for var(y)-classical}, also implies that $\D y$ is absolutely continuous with respect to $\D \ell_\C$ and the existence and uniqueness of $w$ is a consequence of the vectorial Radon-Nikodym theorem \cite[Corollary 4.2, Section VII.4, p. 204]{Lan93}. 
From Lemma \ref{L:int. charact.} we infer that that for every $\mu$-measurable function 
$z : \clsxint{a,b} \function \H$ such that $z(t) \in \C(t)$ for every $t \in \clsxint{a,b}$ we have that
\begin{align}
  & \int_{\clsxint{a,b}} \duality{y(t) - z(t)}{w(t)} \de \D\ell_\C(t) 
      =  \int_{\clsxint{a,b}} \duality{y(t) - z(t)}{\de (w\D\ell_\C)(t)} \notag \\
  &  = \int_{\clsxint{a,b}} \duality{y(t) - z(t)}{\de \D y(t)} 
      = \int_{\clsxint{a,b}} \duality{y(t) - z(t)}{\de (v\mu)(t)} \notag \\
  &  = \int_{\clsxint{a,b}} \duality{y(t) - z(t)}{v(t)} \de \mu(t)  \le 0, \notag
\end{align}
therefore using again Lemma \ref{L:int. charact.} we get \eqref{diff. incl. for Dell}. Finally using 
\cite[Corollary 2.9.9., p. 156]{Fed69} as in formula \eqref{Dl-Lebesgue points} in the proof of Lemma \ref{L:int. charact.}, and exploiting estimate \eqref{estimate for var(y)-classical}, we get that for $\D\ell_\C$-a.e. $t \in \opint{a,b}$ we have
\begin{align}
  \norm{w(t)}{} 
    & = \lim_{h \searrow 0} \frac{1}{\D\ell_\C(\clint{t-h,t+h})} \int_{\clint{t-h,t+h}} \norm{w(\tau)}{} \de \D\ell_\C(\tau) \notag \\
    & = \lim_{h \searrow 0} \frac{\vartot{\D y}(\clint{t-h,t+h})}{\D\ell_\C(\clint{t-h,t+h})}  
     = \lim_{h \searrow 0} \frac{\pV(\C, \clint{a,b})\pV(y, \clint{t-h,t+h})}{(b-a)\pV(\C, \clint{t-h,t+h})} 
         \le \frac{\pV(\C, \clint{a,b})}{b-a} , \notag
\end{align}
and the lemma is completely proved, since $\D\ell_\C(\{0\}) = 0$.
\end{proof}

\begin{Rem}
The previous Lemma can also be proved by using the representation formula \cite[Formula (69)]{Rec16}.
\end{Rem}

Now we can prove our main result in the particular case when the behaviour of the solution is prescribed on a finite number of points.

\begin{Lem}\label{L:case S finite}
Assume that $-\infty < a < b < \infty$, $\C \in \BV^\r(\clsxint{a,b};\Conv_\H)$, $y_0 \in \C(a)$, $S \subseteq \opint{a,b}$, and that for every $t \in S$ we are given $g_t : \H \function \C(t)$ such that 
\begin{equation}\label{sum gt - I finite}
\sum_{t \in S} \norm{g_t - \Id}{\infty,\C(t-)} < \infty.
\end{equation}
If $F \subseteq S$ is a finite set, then there exists a unique $y \in \BV^\r(\clsxint{a,b};\H)$ such that there are a measure 
$\mu : \borel(\clsxint{a,b}) \function \clsxint{0,\infty}$ and a function $v \in \L^1(\mu;\H)$ satisfying
\begin{alignat}{3}
  & y(t) \in \C(t) & \quad &  \forall t \in \clsxint{a,b}, \label{constraint_S finite} \\
  & \D y = v \mu, & \label{Dy decomposition-finite} \\
  & -v(t) \in N_{\C(t)}(y(t)) & \quad & \text{for $\mu$-a.e. $t \in \clsxint{a,b} \setmeno F$}, \label{diff incl_S finite} \\
  & y(t) = g_t(y(t-)) & \quad & \forall t \in F, \label{jump conditions_S finite} \\
  & y(a)  = y_0. & \label{initial condition S finite}
\end{alignat}
Moreover, if $a \le s < t \le b$, we have
\begin{equation}\label{estimate for Var(y)_S finite}
  \pV(y,\clint{s,t})  \le \pV(\C,\clint{s,t}) + 
  \sum_{r \in F \cap \clint{s,t}} \left(\norm{g_r - \Id}{\infty,\C(r-)} - \d_\hausd(\C(r-),\C(r))\right),
\end{equation}
one can take
\begin{equation}\label{formula fo mu-finite}
  \mu = \D\ell_\C + \sum_{t \in S} \norm{g_t - \Id}{\infty, \C(t-)} \delta_{t},
\end{equation}
and for such $\mu$ the function $v$ satisfying \eqref{Dy decomposition-finite}-\eqref{diff incl_S finite} is unique up to
$\mu$-equivalence and we have
\begin{equation}\label{|v| < 1-finite}
  \norm{v(t)}{} \le \max\{1, \pV(\C,\clint{a,b})/(b-a)\} \qquad \text{for $\mu$-a.e. $t \in \clsxint{a,b}$}.
\end{equation}
\end{Lem}

\begin{proof}
It is not restrictive to assume that $\norm{g_t - \Id}{\infty,\C(t-)} \neq 0$ for every $t \in S$. Let us set $t_0 := a$, 
$t_n := b$ and suppose that $F = \{t_1, \ldots, t_{n-1}\}$ for some $n \in \en$ with 
$a = t_0 < t_1 < \cdots < t_{n-1} < t_n = b$. 
Let us call $\C_j$ the restriction of $\C$ to the interval $[t_{j-1},t_j]$. 
Observe that we have
\begin{align}
  \ell_{\C_j}(t) 
    & = t_{j-1} + \frac{t_j-t_{j-1}}{\pV(\C,\clint{t_{j-1},t_j})}\pV(\C,\clint{t_{j-1},t}) \notag \\
    & = t_{j-1} + \frac{t_j-t_{j-1}}{\pV(\C,\clint{t_{j-1},t_j})}\dfrac{\pV(\C,\clint{a,b})}{b-a} \left(\ell_\C(t) - \ell_\C(t_{j-1})\right) \notag
\end{align}
therefore 
\begin{equation}\label{Dellj = c Dell}
\D\ell_{\C_j} = \frac{t_j-t_{j-1}}{\pV(\C,\clint{t_{j-1},t_j})}\dfrac{\pV(\C,\clint{a,b})}{b-a} \D\ell_\C
\end{equation}
and by applying Theorem \ref{T:sweeping process existence} and Lemma \ref{L:mu=Dell} we get that for every 
$j \in \{1, \ldots, n\}$ there is a unique $y^j \in \BV^\r(\clsxint{t_{j-1}, t_j};\H)$ and a unique
$v^j \in \L^1(\D \ell_{\C}, \clsxint{t_{j-1}, t_j};\H)$ such that
\begin{align}
  & y^j(t) \in \C(t)   \quad   \forall t \in \clsxint{t_{j-1}, t_j}, \\
  & \D y^j = v^j \D \ell_{\C}   \quad   \text{on $\borel(\clsxint{t_{j-1}, t_j})$}, \quad \label{Dyj = vj muj} \\
  & -v^j(t) \in N_{\C(t)}(y^j(t))    \quad   \text{for $\D\ell_{\C}$-a.e. $t \in \clsxint{t_{j-1},t_j}$}, \label{diff incl_S finite loc} \\
  & y^j(t_{j-1})  =  \begin{cases}
                             y_0 & \text{if $j=1$} \\
                             g_{t_{j-1}}(y^{j-1}(t_{j-1}-)) & \text{if $j \in \{2, \ldots, n\}$}
                           \end{cases},    
\end{align}
and, using \eqref{|v| < 1-classical} and \eqref{Dellj = c Dell}, we have 
\begin{equation}\label{|v^j| < 1}
  \norm{v^j(t)}{} \le 
   \frac{t_j-t_{j-1}}{\pV(\C,\clint{t_{j-1},t_j})}\dfrac{\pV(f,\clint{a,b})}{b-a}\frac{\pV(\C, \clint{t_{j-1},t_j})}{t_j-t_{j-1}} =
  \dfrac{\pV(f,\clint{a,b})}{b-a}
 \quad \text{for $\D\ell_{\C}$-a.e. $t \in \opint{t_{j-1},t_j}$}.
\end{equation}
Now we define $y : \clsxint{a,b} \function \H$ by setting
\begin{align}
  y(t) := \sum_{j=1}^n \indicator_{\clsxint{t_{j-1},t_j}}(t)y^j(t) 
\end{align}
and $\mu : \borel(\clsxint{a,b}) \function \clsxint{0,\infty}$ by \eqref{formula fo mu-finite}. Observe that $y$ is right continuous and satisfies \eqref{constraint_S finite}, \eqref{jump conditions_S finite}, and  \eqref{initial condition S finite}. Moreover $\mu$ is a positive measure, $\mu(\{0\}) = 0$, and thanks to right continuity of $y$ we have that
\begin{align}
  \pV(y,\clint{s,t}) 
    & =  \sum_{j=1}^{n} \vartot{\D y}(\clsxint{t_{j-1},t_j} \cap\clint{s,t})  \notag \\
    & =  \sum_{j=1}^{n}  \pV(y,\opint{t_{j-1},t_j} \cap\clint{s,t}) + 
            \sum_{t_j \in  F \cap \clint{s,t}} \norm{y(t_{j}) - y(t_{j}-)}{} \notag \\
    & =  \sum_{j=1}^{n} \pV(y^j,\opint{t_{j-1},t_j}\cap\clint{s,t}) + 
            \sum_{F \cap \clint{s,t}}\norm{g_{t_j}(y^{j}(t_{j}-)) - y^{j}(t_{j}-)}{} \notag \\
    & \le \pV(\C,\clint{s,t}) + \sum_{t \in F \cap\clint{s,t}} \norm{g_t - \Id}{\infty,\C(t-)}, \notag
\end{align}
hence $y$ has bounded variation and \eqref{estimate for Var(y)_S finite} holds. Now we consider a set 
$B \in \borel(\clsxint{a,b})$ such that $\mu(B) = 0$. For every $j \in \{1,\ldots,n\}$ we have 
$0 = \mu(B \cap \opint{t_{j-1},t_j}) = \D\ell_\C(B \cap \opint{t_{j-1},t_j})$, thus from \eqref{Dyj = vj muj} we infer that 
$\D y^j(B \cap \opint{t_{j-1},t_j}) = 0$, and, since $y = y^j$ on $\opint{t_{j-1},t_j}$, it follows that $\D y = \D y^{j}$ on 
$\borel(\opint{t_{j-1},t_j})$, so that $\D y(B \cap \opint{t_{j-1},t_j}) = 0$. We also have that 
$0 = \mu(B \cap \{t_j\}) = \norm{g_{t_j} - \Id}{\infty,\C(t_j-)}\delta_{t_j}(B)$ for every $j \in \{1,\ldots,n-1\}$, thus 
$\D y(\{t_j\}) = y(t_j) - y(t_j-) = g_{t_j}(y(t_j-)) - y(t_j-)) = 0$. Therefore we infer that $\D y(B) = 0$ whenever 
$B \in \borel(\clsxint{a,b})$ and $\mu(B) = 0$, so that $\D y$ is $\mu$-absolutely continuous and by the vectorial 
Radon-Nikodym theorem \cite[Corollary 4.2, Section VII.4, p. 204]{Lan93} there exists a unique (up to 
$\mu$-equivalence) function $v \in \L^1(\mu;\H)$ such that $\D y = v \mu$. It follows that on $\borel(\opint{t_{j-1},t_j})$ we have $\D y^j = \D y = v \mu = v \D\ell_\C$, thus $v = v^j$ for $\D\ell_\C$-a.e. $t \in \opint{t_{j-1},t_j}$ and $v$ satisfies \eqref{diff incl_S finite} thanks to \eqref{diff incl_S finite loc}. Let us also observe that for every $t \in F$ we have
\[
  y(t) - y(t-) = \D y(\{t\}) = \int_{\{t\}}v \de \mu = \left(\D\ell_\C(\{t\}) + \norm{g_t - \Id}{\infty,\C(t-)}\right) v(t)
  \qquad \forall t \in F
\]
therefore
\begin{align}
  \norm{v(t)}{} 
     = \frac{\norm{y(t) - y(t-)}{}}{\D\ell_\C(\{t\}) + \norm{g_t - \Id}{\infty,\C(t-)}}
    \le \frac{\norm{g_t(y(t-)) - y(t-)}{}}{\norm{g_t - \Id}{\infty,\C(t-)}} \le 1 \qquad \forall t \in F,
\end{align}
thus formula \eqref{|v| < 1-finite} follows from \eqref{|v^j| < 1}. The uniqueness of $y$ is a consequence of its construction.
\end{proof}

We will need the following weak compactness theorem for measures \cite[Theorem 5, p. 105]{{DieUhl77}}, which we state in a form which is suitable to our purposes.
 
\begin{Thm}\label{measure dunford pettis}
Let $I \subseteq \ar$ be an interval and let $M$ be a subset of the vector space of measures $\nu : \borel(I) \function \H$ with bounded variation endowed with the norm $\norm{\nu}{} := \vartot{\nu}(I)$. Assume that $M$ is bounded. Then $M$ is weakly sequentially precompact if and only if  there is a bounded positive measure 
$\mu : \borel(I) \function \clsxint{0,\infty}$ such that 
\begin{equation}
  \forall \eps > 0 \ \exists \delta> 0 \ \quad :\quad 
  \left( B \in \borel(I),\ \mu(B) < \delta\  \Longrightarrow \ \sup_{\nu \in M} \vartot{\nu}(B) < \eps \right).
\end{equation}
\end{Thm}
It is worthwhile to mention that the above theorem is concerned with the notion of weak convergence of measures (in duality with the space of linear continuous functionals on the space of measures) and not with weak-* convergence of measures (in duality with continuous functions). Anyway, if one does not want an equivalence but only an implication, it is clearly true that the same condition above also provides weakly-* sequential precompactness of $M$.

Theorem \ref{measure dunford pettis} is stated in \cite[Theorem 5, p. 105]{DieUhl77} as a topological precompactness result. An inspection in the proof easily shows that this is actually a sequential precompactness theorem, since an isometric isomorphism reduces it to the well-known Dunford-Pettis weak sequential precompactness theorem in 
$\L^1(\mu;\H)$ (see, e.g., \cite[Theorem 1, p. 101]{DieUhl77}).

We are now in position to prove the following theorem which immediately implies our main result 
Theorem \ref{main thm}.

\begin{Thm}\label{main thm-text}
Assume that $-\infty < a < b < \infty$, $\C \in \BV^\r(\clsxint{a,b};\Conv_\H)$, $y_0 \in \C(a)$, $S \subseteq \opint{a,b}$, and that for every $t \in S$ we are given $g_t : \C(t-) \function \C(t)$ such that $\Lipcost(g_t) \le 1$ and
\begin{equation}\label{sum gt}
  \sum_{t \in S} \norm{g_t - \Id}{\infty,\C(t-)} < \infty.
\end{equation}
Then there exists a unique $y \in \BV^\r(\clsxint{a,b};\H)$ such that there is a measure 
$\mu : \borel(\clsxint{a,b}) \function \clsxint{0,\infty}$ and a function $v \in \L^1(\mu;\H)$ such that
\begin{alignat}{3}
  & y(t) \in \C(t) & \quad & \forall t \in \clsxint{a,b}, \label{constr. [0,T[} \\
  & \D y = v \mu, & \label{Dy =vmu [0,T[} \\
  & -v(t) \in N_{\C(t)}(y(t)) & \quad & \text{for $\mu$-a.e. $t \in \clsxint{a,b} \setmeno S$}, \label{diff. incl. [0,T[} \\
  & y(t) = g_t(y(t-)) & \quad & \forall t \in S, \label{cond. on S, [0,T[} \\
  & y(a)  = y_0. & \label{i.c. [0,T[}
\end{alignat}
Moreover \eqref{estimate for Var(y)_S gen} holds whenever $a \le s < t \le b$.
Finally the function $t \longmapsto \norm{y_1(t) - y_2(t)}{}^2$ is nonincreasing whenever
$y_{0,j} \in \C(a)$, $j = 1, 2$, and $y_j$ is the only function such that there is a measure 
$\mu_j : \borel(\clsxint{a,b}) \function \clsxint{0,\infty}$ and a function $v_j \in \L^1(\mu;\H)$ for which
\eqref{constr. [0,T[}-\eqref{i.c. [0,T[} hold with $y, v, \mu, y_0$ replaced respectively by $y_j, v_j, \mu_j, y_{0,j}$.
\end{Thm}

\begin{proof}
We may assume again that $\norm{g_t - \Id}{\infty,\C(t-)} \neq 0$ for every $t \in S$, thus from \eqref{sum gt} it follows that $S$ is at most countable, and we may assume it contains infinitely many elements, since the finite case is considered in Lemma \ref{L:case S finite}. Let $\mu : \borel(\clsxint{a,b}) \function \clsxint{0,\infty}$ be defined by
\begin{equation}\label{formula for mu-gen}
  \mu = \D\ell_\C + \sum_{t \in S} \norm{g_t - \Id}{\infty, \C(t-)} \delta_{t}.
\end{equation}
If $S = \{s_n\ :\ n \in \en\}$, we set $S_n := \{s_1, \ldots, s_n\}$ for every $n \in \en$ and from Lemma \ref{L:case S finite} it follows that there is a unique $y_n \in \BV^\r(\clsxint{a,b};\H)$ and a unique $v_n \in L^1(\mu;\H)$ satisfying
\begin{alignat}{3}
  & y_n(t) \in \C(t) & \quad & \forall t \in \clsxint{a,b}, \label{y_n in C(t)} \\
  & \D y_n = v_n \mu, \\
  & - v_n(t) \in N_{\C(t)}(y_n(t)) & \quad & \text{for $\mu$-a.e. $t \in \clsxint{a,b} \setmeno S_n$}, \label{v_n in N} \\
  & y_n(t) = g_{t}(y_n(t-)) & \quad & \forall t \in S_n, \label{y_n(t) = g_{t}(y_n(t-))} \\
  & y_n(a) = y_0. \label{c.i. for y_n}
\end{alignat}
Moreover if $a \le s < t \le b$ we have
\begin{align}
  \pV(y_n, \clint{s,t}) 
    \le \pV(\C, \clint{s,t}) + 
    \sum_{r \in S \cap \clint{s,t}} \left(\norm{g_r - \Id}{\infty,\C(r-)}- \d_\hausd(\C(r-),\C(r))\right) < \infty 
    \label{estimate for Var(y_n)}
\end{align}
and 
\begin{equation}\label{|v_n|<1 proof eq}
  \norm{v_n(t)}{} \le \pV(\C,\clint{a,b})/(b-a) \qquad \text{for $\mu$-a.e. $t \in \clsxint{a,b}$}.
\end{equation}
Recall that $S_{n+1} \setmeno S_n = \{s_{n+1}\}$ and assume that $S_n = \{t_1, \ldots, t_{n}\}$ with $t_{j-1} < t_j$ for every $j$, and that  $t_{h-1} < s_{n+1} < t_h$ for some $h \in \{2, \ldots, n\}$ (the cases $s_{n+1} < t_1$ and 
$t_n < s_{n+1}$ are dealt with similarly). Then $y_{n+1}(t) = y_{n}(t)$ for every $t \in \opint{a,s_{n+1}}$, while at 
$t \in \clsxint{s_{n+1}, t_h}$ the distance between $y_n$ and $y_{n+1}$ can be estimated by using 
\eqref{moreau cont. dep.}, Theorem \ref{T:sweeping process existence}, and 
\eqref{y_n in C(t)}--\eqref{y_n(t) = g_{t}(y_n(t-))} as follows:
\begin{align}\label{cauchy estimate on s_(n+1), t_h}
  \sp \norm{y_{n+1}(t) - y_{n}(t)}{}
    & \le \norm{y_{n+1}(s_{n+1}) - y_n(s_{n+1})}{} \notag \\
    & =   \norm{g_{s_{n+1}}(y_{n+1}(s_{n+1}-)) - \Proj_{\C(s_{n+1})}(y_n(s_{n+1}-))}{} \notag \\
    & =   \norm{g_{s_{n+1}}(y_{n}(s_{n+1}-)) - \Proj_{\C(s_{n+1})}(y_n(s_{n+1}-)}{} \notag \\
    & \le \norm{g_{s_{n+1}}(y_{n}(s_{n+1}-)) - y_n(s_{n+1}-)}{}  \notag \\
    & \phantom{\le\ \ } + 
            \norm{y_n(s_{n+1}-)- \Proj_{\C(s_{n+1})}(y_n(s_{n+1}-))}{}   
            \notag \\
    & \le \norm{g_{s_{n+1}} - \Id}{\infty, \C(s_{n+1}-)} + \d_\hausd(\C(s_{n+1}-), \C(s_{n+1})) 
          \qquad \forall t \in \clsxint{s_{n+1}, t_h}.\end{align}
If $h < n$ and $t \in \clsxint{t_h, t_{h+1}}$, from \eqref{moreau cont. dep.}, 
\eqref{y_n in C(t)}-\eqref{y_n(t) = g_{t}(y_n(t-))}, and \eqref{cauchy estimate on s_(n+1), t_h}, we infer that
\begin{align}
  \norm{y_{n+1}(t) - y_{n}(t)}{} 
    & \le \norm{g_{t_h}(y_{n+1}(t_h-)) - g_{t_h}(y_{n}(t_h-))}{} \notag \\
    & \le \norm{y_{n+1}(t_h-) - y_{n}(t_h-)}{} \notag \\
    & = \lim_{t \to t_h-} \norm{y_{n+1}(t) - y_{n}(t)}{} \notag \\
    & \le \norm{g_{s_{n+1}} - \Id}{\infty, \C(s_{n+1}-)} + \d_\hausd(\C(s_{n+1}-), \C(s_{n+1})), \notag
\end{align}
and iterating this procedure we get the same estimate for $t \in \clsxint{t_k, t_{k+1}}$ and $h \le k < n$, thus
\[
  \norm{y_{n+1} - y_{n}}{\infty} \le \norm{g_{s_{n+1}} - \Id}{\infty, \C(s_{n+1}-)} + \d_\hausd(\C(s_{n+1}-), \C(s_{n+1})),
\]
and
\begin{align}
  \sum_{n=1}^\infty \norm{y_n - y_{n+1}}{\infty} 
    & \le \sum_{n=1}^\infty \Big( \norm{g_{s_{n+1}} - \Id}{\infty, \C(s_{n+1}-)} + \d_\hausd(\C(s_{n+1}-), \C(s_{n+1})) \Big) \notag \\
   & \le \sum_{t \in S} \norm{g_{t} - \Id}{\infty, \C(t-)} + \pV(\C,\clsxint{a,b}) < \infty \notag
\end{align}
i.e. $y_n$ is uniformly Cauchy and there exists $y : \clsxint{a,b} \function \H$ such that
\begin{equation}\label{yn -> y unif}
  y_n \to y \qquad \text{uniformly on $\clsxint{a,b}$}.
\end{equation}
Moreover $y \in \BV(\clsxint{a,b};\H)$ and \eqref{estimate for Var(y)_S gen} holds by \eqref{estimate for Var(y_n)} and the semicontinuity of the variation w.r.t. to the pointwise convergence, and $y$ satisfies \eqref{constr. [0,T[} by virtue of the closedness of $\C(t)$. Conditions \eqref{cond. on S, [0,T[}-\eqref{i.c. [0,T[} are trivially satisfied because for every 
$t \in S$ the sequence $y_n(t)$ is definitively constant, equal to $g_t(y(t-))$. Let us observe that thanks to 
\eqref{|v_n|<1 proof eq} we have that
\[
  \vartot{\D y_n}(B) = \int_B \norm{v_n(t)}{} \de \mu(t) \le 
  \int_B \frac{\pV(\C,\clint{a,b})}{b-a} \de \mu(t) = 
  \frac{\pV(\C,\clint{a,b})}{b-a}\mu(B) \qquad \forall B \in \borel(\clsxint{a,b})
\]
hence by the Dunford-Pettis Theorem \ref{measure dunford pettis} for measures and from \cite[Lemma 7.1]{KopRec16} we infer that
\begin{equation}\label{Dyn -> Dy}
  \D y_n \convergedeb \D y
\end{equation}
Since $\norm{v_n}{\L^2(\mu;\H)} \le \pV(\C,\clint{a,b})\sqrt{\mu(\clsxint{a,b})}/(b-a) < \infty$, there exists 
$v \in \L^2(\mu;\H)$ such that, up to subsequences, 
\begin{equation}\label{v_n -> v}
  v_n \convergedeb v \qquad \text{in $\L^2(\mu;\H)$ as $n \to \infty$},
\end{equation} 
and thanks to \eqref{v_n in N} and Lemma \ref{L:int. charact.} we have that
\begin{equation}
  \int_{\clsxint{a,b}\!\!\ \setmeno\!\ S} \duality{y_n(t) - z(t)}{v_n(t)} \de \mu(t) \le 0 \qquad \forall n \in \en,
\end{equation}
therefore taking the limit as $n \to \infty$ from \eqref{yn -> y unif} and \eqref{v_n -> v} we infer that
\begin{equation}
  \int_{\clsxint{a,b}\!\!\ \setmeno\!\ S} \duality{y(t) - z(t)}{v(t)} \de \mu\le 0
\end{equation}
which, by Lemma \ref{L:int. charact.}, is equivalent to \eqref{diff. incl. [0,T[}. 
Estimate \eqref{estimate for Var(y)_S gen} follows from \eqref{estimate for Var(y_n)} and from the lower semicontinuity of the variation.
If $\phi : \clsxint{a,b} \function \H$ is an arbitrary bounded Borel function then 
$\mu \longmapsto \int_{\clsxint{a,b}} \duality{\phi(t)}{\de \mu(t)}$ is a continuous linear functional on the space of measures with bounded variation and we have
\[
  \lim_{n \to \infty} \int_{\clsxint{a,b}} \duality{\phi(t)}{\de \D y_n(t)} = \int_{\clsxint{a,b}} \duality{\phi(t)}{\de \D y(t)}.
\]
On the other hand we have
\[
  \int_{\clsxint{a,b}} \duality{\phi(t)}{\de \D y_n(t)} = \int_{\clsxint{a,b}} \duality{\phi(t)}{v_n(t)}\de \mu(t)
   \qquad \forall n \in \en,
\]
thus taking the limit as $n \to \infty$ thanks to \eqref{Dyn -> Dy} and \eqref{v_n -> v} we get
\[
  \int_{\clsxint{a,b}} \duality{\phi(t)}{\de \D y(t)} = \int_{\clsxint{a,b}} \duality{\phi(t)}{v(t)}\de \mu(t)
\]
therefore the arbitrariness of $\phi$ yields $\D y = v \mu$ (cf., e.g., \cite[Proposition 35, p. 326]{Din67}). Then
\eqref{Dy =vmu [0,T[} is proved and the existence part of the theorem is done. If $B \in \borel(\clsxint{a,b})$ then, by \cite[Proposition 2]{Mor76} and by the monotonicity of the normal cone we have
\begin{align}
  \int_{B\!\  \setmeno\!\  S} \de \D\!\ (\norm{y_1(\cdot) - y_2(\cdot)}{}^2) 
    & \le 2 \int_{B\!\ \setmeno\!\  S} \duality{y_1 - y_2}{\de \D\!\ (y_1 - y_2)} \notag \\
    & = 2 \int_{B\!\  \setmeno\!\  S} \duality{y_1(t) - y_2(t)}{v_1(t) - v_2(t)} \de \mu(t) \le 0,
\end{align}
while if $t \in S$ then we have
\begin{align}
  \D\!\ (\norm{y_1(\cdot) - y_2(\cdot)}{}^2)(\{t\})
    & =   \norm{y_1(t) - y_2(t)}{}^2 - \norm{y_1(t-) - y_2(t-)}{}^2 \notag \\
    & =   \norm{g_t(y_1(t-)) - g_t(y_2(t-))}{} - \norm{y_1(t-) - y_2(t-)}{}^2 \notag \\ 
    & \le \norm{y_1(t-) - y_2(t-)}{} - \norm{y_1(t-) - y_2(t-)}{}^2 = 0,
\end{align}
therefore for every $B \in \borel(\clsxint{a,b})$ we find
\begin{align}
  \D\!\ (\norm{y_1(\cdot) - y_2(\cdot)}{}^2)(B) 
    & = \D\!\ (\norm{y_1(\cdot) - y_2(\cdot)}{}^2)(B \setmeno S) + 
            \D\!\ (\norm{y_1(\cdot) - y_2(\cdot)}{}^2)(B \cap S) \notag \\
    & = \int_{B\!\ \setmeno\!\  S} \de \D\!\ (\norm{y_1(\cdot) - y_2(\cdot)}{}^2) + 
          \sum_{t \in B \cap S} \D\!\ (\norm{y_1(\cdot) - y_2(\cdot)}{}^2)(\{t\}) \le 0 \notag
\end{align}
which implies that $t \longmapsto \norm{y_1(t) - y_2(t)}{}^2$ is nonincreasing and leads to the uniqueness of the solution. As a consequence the whole sequence $y_n$ converges uniformly to $y$.
\end{proof}


\section{Applications}\label{S:applications}

In this section we discuss some consequences and particular cases of Theorem \ref{main thm}.
 

\subsection{Sweeping processes with arbitrary $\BV$ driving set and prescribed behaviour on jumps}

We first consider the case of a sweeping processes with prescribed behaviour on jumps, where the driving moving set is not assumed to be right continuous.

\begin{Thm}\label{T:gen sweep gen BV}
Assume that $-\infty < a < b < \infty$, $\C \in \BV(\clint{a,b};\Conv_\H)$, $y_0 \in \H$, and that for every 
$t \in \discont(\C)  \cup \{a\}$ we are given $g_t^\l : \C(t-) \function \C(t)$ and $g_t^\r : \C(t) \function \C(t+)$ such that 
\begin{equation}\label{Lip cond gt-l-r}
  \Lipcost(g_t^\l) \le 1, \quad\Lipcost(g_t^\r) \le 1 \qquad \forall t \in \discont(\C) \!\ \setmeno\!\ \{a\}
\end{equation}
and
\begin{equation}\label{sum gt-l-r}
  \sum_{t \in \discont(\C)} \norm{g_t^\l - \Id}{\infty,\C(t-)} < \infty, \quad \sum_{t \in \discont(\C)} \norm{g_t^\r - \Id}{\infty,\C(t)} < \infty.
\end{equation}
Then there exists a unique $y \in \BV(\clint{a,b};\H)$ such that there is a measure
$\mu : \borel(\clint{a,b}) \function \clsxint{0,\infty}$ and a function $v \in \L^1(\mu;\H)$ for which
\begin{align}
  & y(t) \in \C(t) \qquad \forall t \in \clint{a,b}, \label{constr. very gen} \\
  & \D y = v \mu, \label{Dy very gen} \\
  & -v(t) \in N_{\C(t)}(y(t)) \qquad \text{for $\mu$-a.e. $t \in \cont(\C)$}, \\
  & y(t) = g_t^\l((y(t-))), \quad
    y(t+) = g_t^\r(g_t^\l((y(t-))) \qquad \forall t \in \discont(\C)\!\ \setmeno \!\ \{a\}, \\
  & y(a) = g_a^\l(y_0), \quad y(a+) = g_a^\r(a)(y(a)). \label{i.c. very  gen}
\end{align}
Moreover if $a \le s < t \le b$ then
\begin{align}
 & \pV(y,\clint{s,t})  \le \pV(\C,\clint{s,t}) \notag \\
 & + \sum_{\sigma \in \discont(\C) \cap \clint{s,t}} 
  \Big(\norm{g_\sigma^\l - \Id}{\infty,\C(\sigma-)} + 
         \norm{g_\sigma^\r - \Id}{\infty,\C(\sigma)} -
              \d_\hausd(\C(\sigma-),\C(\sigma)) + \d_\hausd(\C(\sigma),\C(\sigma+))\Big). \notag
\end{align}
\end{Thm}

\begin{proof}
We can apply Theorem \ref{main thm} with $g_t := g_t^\r \circ g_t^\l$ and find a unique function 
$\hat{y} \in \BV^\r(\clint{a,b};\H)$ such that there is a measure $\hat{\mu} : \borel(\clint{a,b}) \function \clsxint{0,\infty}$ and a function 
$\hat{v} \in \L^1(\mu;\H)$ for which
\begin{align}
  & \hat{y}(t) \in \C(t+) \qquad \forall t \in \clint{a,b} \label{constr yhat} \\
  & \D \hat{y} = \hat{v} \hat{\mu}, \\
  & - \hat{v}(t) \in N_{\C(t+)}(\hat{y}(t)) = N_{\C(t)}(\hat{y}(t)) \qquad \text{for $\hat{\mu}$-a.e. $t \in \cont(\C)$}, \\
  & \hat{y}(t) = g_t^\r(g_t^\l(\hat{y}(t-))) \qquad \forall t \in \discont(\C) \setmeno \{a\}, \\
  & \hat{y}(a) = g_a^\r(y_0). \label{i.c. yhat}
\end{align}
Then the theorem is satisfied if we take $y : \clint{a,b} \function \H$, $\mu : \borel(\clint{a,b}) \function \clsxint{0,\infty}$, and $v: \clint{a,b} \function \H$ defined by
\[
  y(t) := 
    \begin{cases}
      \hat{y}(t) & \text{if $t \in \cont(\C)$} \\
      g_t^\l(\hat{y}(t-)) & \text{if $t \in \discont(\C)$} 
    \end{cases},
\]
$\mu:= \hat{\mu}$, and
\[
  v(t) := 
  \begin{cases}
  \hat{v}(t) & \text{if $t \in \cont(\C)$} \\
  0 & \text{if $t \in \discont(\C)$ and $\hat{\mu}(\{t\}) = 0$}
  \ \\
  \dfrac{y(t+) - y(t-)}{\hat{\mu}(\{t\})} & \text{if $t \in \discont(\C)$ and $\hat{\mu}(\{t\}) \neq 0$} 
  \end{cases}.
\]
(observe that the first condition in \eqref{sum gt-l-r} ensures that $y \in \BV(\clint{a,b};\H)$).
\end{proof}

\begin{Rem}
Let us observe that we can actually prove a result which is more general than Theorem \ref{T:gen sweep gen BV}: indeed we can prescribe the behavior of the solution $y(t)$ also on a countable set of points $t$ where $\C$ is continuous. In order to do that we need to assume that the families $g_t^\l : \C(t-) \function \C(t)$ and 
$g_t^\r : \C(t) \function \C(t+)$ are indexed by $t \in S$, where $S \subseteq \clint{a,b}$, $\Lipcost(g_t^\l) \le 1$ and 
$\Lipcost(g_t^\r) \le 1$ for every $t \in S$, and $\sum_{t \in S} \norm{g_t^\l - \Id}{\infty,\C(t-)} < \infty$ and
$\sum_{t \in S} \norm{g_t^\r - \Id}{\infty,\C(t)} < \infty$. Therefore it follows that if $\C \in \BV(\clint{a,b};\Conv_\H)$, then there exists a unique $y \in \BV(\clint{a,b};\H)$ such that there is a measure 
$\mu : \borel(\clint{a,b}) \function \clsxint{0,\infty}$ and a function $v \in \L^1(\mu;\H)$ for which \eqref{constr. very gen}, \eqref{Dy very gen}, and \eqref{i.c. very  gen} hold together with
\begin{align}
  & -v(t) \in N_{\C(t)}(y(t)) \qquad \text{for $\mu$-a.e. $t \in \clint{a,b} \setmeno S$}, \\
  & y(t) = g_t^\l((y(t-))), \quad
    y(t+) = g_t^\r(g_t^\l((y(t-))) \qquad \forall t \in S\!\ \setmeno \!\ \{a\}.
\end{align}
Observe that in this case $S$ is a fortiori at most countable and $y$ may jump even when $\C$ does not jump.
\end{Rem}


\subsection{Sweeping processes with arbitrary $\BV$ driving moving set.}

Another consequence of Theorem \ref{T:gen sweep gen BV} is the existence and uniqueness theorem for sweeping processes with arbitrary $\BV$ driving moving set.

\begin{Thm}
Assume that $-\infty < a < b < \infty$, $\C \in \BV(\clint{a,b};\Conv_\H)$ and $y_0 \in \H$. Then there exists a unique 
$y \in \BV(\clint{a,b};\H)$ such that there is a measure $\mu : \borel(\clint{a,b}) \function \clsxint{0,\infty}$ and a function $v \in \L^1(\mu;\H)$ for which
\begin{align}
  & y(t) \in \C(t), \label{constr. arbBV sweep} \\
  & \D y = v \mu, \\
  & -v(t) \in N_{\C(t)}(y(t)) \qquad \text{for $\mu$-a.e. $t \in \cont(\C)$}, \\
  & y(t) = \Proj_{\C(t)}(y(t-)), \qquad
   y(t+) = \Proj_{\C(t+)}(y(t)) \qquad \forall t \in \discont(\C) \!\ \setmeno \!\ \{a\} \\
  & y(a) = \Proj_{\C(a)}(y_0), \quad y(a+) = \Proj_{\C(a+)(y(a))}. \label{i.c. arbBV sweep}
\end{align}
Moreover $\pV(y, \clint{s,t}) \le \pV(\C, \clint{s,t})$ whenever $a \le s < t \le b$.
\end{Thm}

\begin{proof}
It is enough to apply Theorem \ref{T:gen sweep gen BV} with $g_t^\l := \Proj_{\C(t+)}$ and $g_t^\r := \Proj_{\C(t)}$.
\end{proof}


\subsection{The play operator}\label{SS:play}

The play operator is the solution operator of the sweeping process driven by a moving set $\C(t)$ with constant shape, i.e. $\C(t) = u(t) - \Z$, where $u \in \BV^\r(\clint{a,b};\H)$ and $\Z \in \Conv_\H$. In the following result we restate here the existence Theorem \ref{T:sweeping process existence} in this particular case by using the integral formulation of Lemma \ref{L:int. charact.} and we collect some other well-known results (see also \cite{KreLau02}, where the Young integral is used, and \cite[Section 5]{KopRec16} containing a slightly different integral formulation).

\begin{Thm}
Assume that $-\infty < a < b < \infty$, $\Z \in \Conv_\H$, $u \in \BV^\r(\clint{a,b};\H)$ and $z_0 \in \Z$. Then there exists a unique $y \in \BV^\r(\clint{a,b};\H)$ such that
\begin{align}
  & y(t) \in u(t) - \Z \qquad \forall t \in \clint{a,b}, \label{play constr} \\
  & \int_{\clint{a,b}} \duality{z(t) - u(t) + y(t)}{\de\D y(t)} \le 0 
       \qquad \text{for every $\mu$-measurable $z : \clint{a,b} \function \Z$}, \label{play ineq} \\
  & u(0) - y(0) = z_0. \label{play i.c.}
\end{align}
The solution operator $\P : \BV^\r(\clint{a,b};\H) \times \Z \function \BV^\r(\clint{a,b};\H)$ associating with 
$(u,z_0) \in \BV^\r(\clint{a,b};\H) \times \Z$ the unique function $y = \P(u,z_0)$ satisfying 
\eqref{play constr}-\eqref{play i.c.} is called \emph{play operator} and it is \emph{rate independent}, i.e. 
\begin{equation}\label{rate ind}
  \P(u \circ \psi, z_0) = \P(u, z_0) \circ \psi \qquad \forall u \in \BV^\r(\clint{a,b};\H)
\end{equation}
whenever $\psi \in \Czero(\clint{a,b};\clint{a,b})$ is nondecreasing and surjective. We have that 
$\P(\Lip(\clint{a,b};\H) \times \Z)$ $\subseteq$ $\Lip(\clint{a,b};\H)$ and if $u \in \Lip(\clint{a,b};\H)$ then $\P(z_0, u) = y$ is the unique function satisfying \eqref{play constr}, \eqref{play i.c.}, and
\begin{equation}
  \lduality{z(t) - u(t) - y(t)}{y'(t)} \le 0 \qquad 
  \text{for $\leb^1$-a.e. $t \in \clint{a,b},\ \forall z \in \Z$}. \label{play ineq Lip} 
\end{equation}
The restriction $\P : \Lip(\clint{a,b};\H) \times \Z \function \Lip(\clint{a,b};\H)$ is continuous if $\Lip(\clint{a,b};\H)$ is endowed with both the $\BV$-norm metric 
\[
  \d_\BV(v,w) := \norm{v - w}{\infty,\clint{a,b}} + \pV(v-w;\clint{a,b}),
\]
and the strict metric
\[
  \d_s(v,w) := \norm{v - w}{\L^1(\leb^1;\H)} + |\pV(v;\clint{a,b}) - \pV(w;\clint{a,b})|,
\]
and the operator $\P : \BV^\r(\clint{a,b};\H) \times \Z \function \BV^\r(\clint{a,b};\H)$ is its unique continuous extension w.r.t. the topology induced by $\d_\BV$.
\end{Thm}

Let us observe that the integral formulation \eqref{play constr}-\eqref{play i.c.} is a direct consequence of Theorem 
\ref{T:sweeping process existence} and Lemma \ref{L:int. charact.}. Formulas \eqref{rate ind} and 
\eqref{play ineq Lip} are well-known (see, e.g., \cite[Section 3i, Section 3c]{Mor77} and \cite[Proposition 3.9, p. 33]{Kre97}). The continuity of the restriction of $\P$ w.r.t. $\d_\BV$ is proved in \cite[Theorem 3.12, p. 34]{Kre97}, while its continuity w.r.t $\d_s(v,w)$ is proved in \cite[Theorem 5.5]{Rec11}. The $\BV$-norm continuity of $\P$ on 
$\BV^\r(\clint{a,b};\H) \times \Z$ is proved in \cite[Theorem 3.3]{KopRec16}.

The operator $\P : \BV^\r(\clint{a,b};\H) \times \Z \function \BV^\r(\clint{a,b};\H)$ in general is not continuous w.r.t. to the strict metric $\d_s$ (cf. \cite[Thereom 3.7]{Rec11}) but its restriction 
$\P : \Lip(\clint{a,b};\H) \times \Z \function \Lip(\clint{a,b};\H)$ can be continuously extended to another operator if 
$\BV^\r(\clint{a,b};\H)$ is endowed with the strict topology in the domain, and with the $\L^1$-topology in the codomain. Let us recall the precise result proved in \cite[Thereom 3.7]{Rec11}.

\begin{Thm}\label{T:Pbar}
The operator $\P : \Lip(\clint{a,b};\H) \times \Z \function \Lip(\clint{a,b};\H)$ admits a unique continuous extension 
$\overline{\P} : \BV^\r(\clint{a,b};\H) \times \Z \function \BV^\r(\clint{a,b};\H)$ if $\BV^\r(\clint{a,b};\H)$ is endowed with the strict topology $\d_s$ in the domain, and with the $\L^1(\leb^1;\H)$-topology in the codomain. We have 
\begin{equation}\label{Pbar}
  \overline{\P}(u, z_0) = \P(\utilde,z_0) \circ \ell_u \qquad \forall u \in \BV^\r(\clint{a,b};\H),
\end{equation}
where $\utilde \in \Lip(\clint{a,b};\H)$ is the arc length reparametrization introduced in Proposition \ref{D:reparametrization}. In general $\overline{\P} \neq \P$.
\end{Thm}

The meaning of the extension $\overline{\P}$ and of formula \eqref{Pbar} is clear: we reparametrize by the arc length the function $u = \utilde \circ \ell_u$ and we apply the play operator to the Lipschitz reparametrization $\utilde$, which is a segment on the jump sets $\clint{\ell_u(t-), \ell_u(t+)}$.  Then we ``throw away'' the jump sets from $\P(\utilde,z_0)$ by reinserting $\ell_u$ and we obtain $\overline{\P}(u,z_0) = \P(\utilde,z_0) \circ \ell_u$. A further motivation to this procedure is the rate independence of $\P$ (but observe that $\ell_u$ is not continuous) and we could also say that we are filling in the jumps of $u$ with a segment traversed with ``infinite velocity''. The continuity property of $\overline{\P}$ in Theorem \ref{T:Pbar} confirms this interpretation. Now we are going to show that $\overline{\P}(u,z_0)$ can also be obtained as the solution of a sweeping processes with a suitable prescribed behaviour on jumps.

\begin{Thm}
Assume that $\Z \in \Conv_\H$, $u \in \BV^\r(\clint{a,b};\H)$ and $z_0 \in \Z$. 
For every $x, y  \in \H$ let $\seg_{x,y} : \clint{0,1} \function \H$ be defined by $\seg_{x,y}(t) := (1-t)x + ty$, 
$t \in \clint{0,1}$. Then there exists a unique $y \in \BV^\r(\clint{a,b};\H)$ such that
\begin{align}
  & y(t) \in u(t) - \Z \qquad \forall t \in \clint{a,b} \label{Pbar constr g} \\
  & \int_{\cont(u)} \duality{z(t) - u(t) + y(t)}{\de\D y(t)} \le 0 \quad 
                       \text{for every $\mu$-measurable $z :\clint{a,b} \function \Z$}, 
      \label{P ineq g} \\
  & y(t) = \P(\seg_{u(t-),u(t)},u(t-) - y(t-))(1) \qquad \forall t \in \discont(\C), \label{P g cond.} \\
  & y(0) = y_0. \label{Pbar i.c. g}
\end{align}
Moreover if $u \in \BV^\r(\clint{a,b};\H)$, $z_0 \in \Z$, $s, t \in \clint{a,b}$, $s < t$, then we have
\begin{equation}\label{y-g = Pbar}
  y = \P(\utilde,z_0) \circ \ell_u = \overline{\P}(u,z_0),
\end{equation}
and
\begin{equation}\label{BV est for Pbar}
  \pV(\Pbar(u),\clint{s,t}) \le \pV(u,\clint{s,t}).
\end{equation}
\end{Thm}

\begin{proof}
We set $\C(t) := u(t) - \Z$, $y_0 := u(0) - z_0$, $S := \discont(u)$, and 
\[
  g_t(x) := \P(\seg_{u(t-),u(t)}, u(t-) -x)(1), \qquad t \in S.
\] 

If $t \in \discont(u)$ and $x \in \C(t-)$ then $g_t(x)$ is the solution of the sweeping process driven by 
$\K_u(\tau) := (1-\tau)u(t-) - \tau u(t) - \Z$, $\tau \in \clint{0,1}$, and having $x$ as initial condition, thus from
\eqref{estimate for var(y)-classical} we infer that
$\pV(\P(\seg_{u(t-),u(t)}, u(t-) - x), \clint{0,1}) \le \pV(\K_u, \clint{0,1}) = \norm{u(t) - u(t-)}{}$ and we have
\begin{align}
  \sum_{t \in \discont(u)}\norm{g_{t} - \Id}{\infty,\C(t-)}
  & = \sum_{t \in \discont(u)} \sup_{x \in \C(t-)} \norm{g_{t}(x) - x}{} \notag \\
    & = \sum_{t \in \discont(u)} \sup_{x \in \C(t-)}\norm{\P(\seg_{u(t-),u(t)}, u(t-) - x)(1) - x}{} \notag \\
    & \le \sum_{t \in \discont(u)} \sup_{x \in \C(t-)} \pV(\P(\seg_{u(t-),u(t)}, u(t-) - x), \clint{0,1}) \notag \\
    & \le \sum_{t \in \discont(u)} \norm{u(t) - u(t-)}{} \le \pV(u, \clint{0,T}) \notag
\end{align}
which together with \eqref{estimate for Var(y)_S gen} implies \eqref{BV est for Pbar}.
Moreover thanks to \eqref{moreau cont. dep.} we have that $\norm{g_{t}(x_1) - g_{t}(x_2)}{} \le \norm{x_1 - x_2}{}$ therefore we can apply Theorem \ref{main thm} and infer the existence of a unique $y$ satisfying 
\eqref{Pbar constr g}--\eqref{Pbar i.c. g}. Now we show that \eqref{y-g = Pbar} holds. Thanks to the chain rule 
\cite[Theorem A.7]{Rec11} we have that $\D\Pbar(u,z_0) = w \D\ell_u$ with 
\[
  w(t) =
  \begin{cases}
    (\P(\utilde),z_0)'(t) & \text{if $t \in \cont(u)$} \\
    \ \\
    \dfrac{\P(\utilde,z_0)(\ell_u(t)) - \P(\utilde,z_0)(\ell_u(t-))}{\ell_u(t) - \ell_u(t-))} & \text{if $t \in \discont(u)$}
  \end{cases}
\]
and if we set
\begin{equation}\label{Z}
  N := \{\sigma \in \clint{a,b}\ :\ \duality{z - \utilde(\sigma) + \P(\utilde,z_0)(\sigma)}{(\P(\utilde,z_0))'(\sigma)} > 0 
      \text{ for some } z \in \Z\} \notag
\end{equation} 
from \eqref{play ineq Lip} we deduce that $\leb^1(N) = 0$,
hence, thanks to \cite[Lemma A.5]{Rec11}, we have that
\begin{align}
  & \D\ell_{\C}(\{t \in \cont(\ell_{\C})\ :\ 
                \duality{z - \utilde(\ell_u(t)) + \P(\utilde,z_0)(\ell_u(t))}{(\P(\utilde),z_0)'(\ell_u(t))} > 0
                \text{ for some } z \in \Z\}) \notag \\
  = & \D\ell_{\C}(\{t \in \cont(\ell_{\C})\ :\ \ell_{\C}(t) \in N\}) = \leb^{1}(N) = 0. \notag
\end{align}
This implies that if $z : \clint{a,b} \function \H$ is $\mu$-measurable and $z(t) \in \Z$ for every $t \in \clint{a,b}$, then \begin{align}
  & \int_{\cont(u)} \duality{z(t) - u(t) + \P(\utilde,z_0)(\ell_u(t))}{\de \D\!\ (\P(\utilde,z_0) \circ \ell_u)}  \notag \\
     & = \int_{\cont(u)} \duality{z(t) - u(t) + \P(\utilde,z_0)(\ell_u(t))}{(\P(\utilde,z_0)'(\ell_u(t))} \de \D\ell_\C(t) \le 0 \notag
\end{align}
and \eqref{P ineq g} is satisfied. Finally \eqref{P g cond.} is trivially satisfied by $\Pbar(u,z_0)$.
\end{proof}



\end{document}